\documentclass[12pt]{article}

\usepackage{amsmath, amssymb, amsthm, mathrsfs}
\usepackage{hyperref}
\usepackage{geometry}
\newtheorem{theorem}{Theorem}[section]

\newtheorem{definition}{Definition}[section]
\newtheorem{prop}{Proposition}[section]

\newtheorem{lemma*}{Lemma A.}[]

\title{Cohomologically Calibrated Affine Connections and Forced Irreducibility}
\author{Alexander Pigazzini, Magdalena Toda}
\date{}

\begin{document}

\maketitle

\begin{abstract}
We establish a principle of forced geometric irreducibility on product manifolds. We prove that for any product manifold $M=M_1\times M_2$, a cohomologically calibrated affine connection, $\nabla^{\mathcal{C}}$, is necessarily holonomically irreducible, provided its calibration class $[\omega] \in H^3(M;\mathbb{R})$ is mixed. The core of the proof relies on Hodge theory; we show that the algebraic structure of the harmonic part of the torsion generates non-zero off-diagonal components in the full Riemann curvature tensor, which cannot be globally cancelled. This non-cancellation is formally proven via an integral argument. We illustrate the main theorem with explicit constructions on $S^2\times \Sigma_g$, showing that this result holds even in special cases where the Ricci tensor is diagonal, such as the Einstein-calibrated connection. Finally, we briefly discuss speculative analogies between forced irreducibility and quantum entanglement.
\end{abstract}

\section{Introduction and preliminaries}

The de Rham Decomposition Theorem is a cornerstone of Riemannian geometry, providing a deep connection between the geometric and topological structure of a manifold. In its standard formulation, it states that a simply connected Riemannian manifold $(M, g)$ that is a metric product $(M_1, g_1) \times (M_2, g_2)$ is also a geometric product. This principle of geometric separability, which extends to non-simply connected cases by considering their universal covers, implies that the holonomy group $\text{Hol}(g)$ of a product manifold typically acts reducibly on the tangent space $T_pM = T_pM_1 \oplus T_pM_2$. Consequently, the geometry on $M_1$ is independent of the geometry on $M_2$.
In this paper, we present a novel framework that provides a systematic method to construct irreducible geometries on such product manifolds. Our approach moves beyond the torsion-free Levi-Civita connection by introducing a family of affine connections $\nabla^{\mathcal{C}}$ whose torsion $T$ is constrained by a global topological invariant. We call this framework \emph{Cohomologically Calibrated affine connections} (see \cite{Pigazzini}, \cite{Pigazzini-Toda}).
\\
The central result of this work is to demonstrate that this calibration mechanism forces the resulting geometry to be irreducible, provided the calibrating cohomology class is "mixed" with respect to the product structure. This geometric "entanglement" is not a coincidence of special solutions (e.g., Einstein metrics) but a general and robust consequence of the calibration principle itself. We show that the very structure of a calibrated torsion tensor introduces curvature components that inextricably link the tangent subspaces of the product factors, making it impossible for the holonomy group to preserve the decomposition.
\\
This work establishes a new principle: global topology, when used to calibrate a local geometric structure like a connection, can override the metric's tendency to decompose, yielding a rich class of new, irreducible geometries.

\section{The Framework of Cohomological Calibration}

We begin by recalling the central objects of our theory.

\begin{definition} A Cohomologically Calibrated Geometric Structure is a triplet $([M, g], [\omega])$, where:
\begin{itemize}
\item $(M, g)$ is a compact, oriented Riemannian manifold.
\item$[\omega]$ is a non-trivial de Rham cohomology class in $H^3(M; \mathbb{R})$.
\end{itemize}    
\end{definition} 

\noindent For further details we recommend \emph{Remark 3.1} in \cite{Pigazzini}  and \emph{Subsection 2.5} in \cite{Pigazzini-Toda}. 

\begin{definition} Let $([M, g], [\omega])$ be a 3-calibrated structure. The \emph{family of admissible connections} $\mathcal{T}_\omega$ is the set of the connections $\nabla^{\mathcal{C}}=\nabla^{LC}+T$ with type (1,2) tensors $T$ on $M$ such that:
\begin{itemize}
\item The construction operates under the explicit assumption of a fixed Riemannian background metric  $g$, which is used to define norms, orthogonality, musical isomorphisms, etc., while, in general, $\nabla^{\mathcal{C}}$ is not required to be compatible with the metric ($\nabla^{\mathcal{C}} g \neq 0$).
\item $T$ is totally antisymmetric, which establishes an isomorphism with the space of 3-forms $\Omega^3(M)$. We denote the 3-form associated with $T$ as $T^\flat$, where $T^\flat(X,Y,Z) := g(T(X,Y),Z)$.
\item The cohomology class of $T^\flat$ is $[\omega]$. That is, $[T^\flat] = [\omega]$ with $[\omega] \neq 0$.
\end{itemize}

\noindent By the Hodge Decomposition Theorem, any $T^\flat$ in this family can be uniquely written as $T^\flat = \omega + d\eta +\delta \mu$, where $\omega$ is the unique harmonic representative of the class $[\omega]$, $d\eta$ and $\delta \mu$ are 3-form (exact and co-exact, respectively). The family $\mathcal{T}_\omega$ is therefore an affine space modeled on the space of exact 3-forms.
\end{definition}

\begin{definition} An admissible connection $\nabla^{\mathcal{C}}$ is an affine connection of the form $\nabla^{\mathcal{C}}= \nabla^{LC} + T$, where $\nabla^{LC}$ is the Levi-Civita connection of $g$ and $\nabla^{\mathcal{C}} \in \mathcal{T}_\omega$.
\end{definition}

\noindent This requirement ties the local differential geometry of $\nabla^{\mathcal{C}}$ to the global topology invariant of $M$.  

We focus on the case in which $M$ is a Riemannian product
\[
M = M_1 \times M_2
\]
\noindent and $[\omega]$ has a \emph{mixed} component in its K\"unneth decomposition.

\section{Forced Irreducibility}

Our main result establishes that the principle of cohomological calibration can force a reducible metric structure to support an irreducible geometry. Let $M = M_1 \times M_2$ be a product of compact, oriented Riemannian manifolds with the product metric $g = g_1 + g_2$. The tangent space decomposes at each point as $T_pM = V_1 \oplus V_2$, where $V_1 = T_pM_1$ and $V_2 = T_pM_2$. The holonomy of $(M, \nabla^{\mathrm{LC}})$ is reducible, preserving this decomposition.

\begin{definition}
A cohomology class $[\omega] \in H^3(M; \mathbb{R})$ is mixed if it has a non-trivial projection onto a summand $H^p(M_1) \otimes H^q(M_2)$ where both $p > 0$ and $q > 0$.
\end{definition}

\begin{theorem} \label{thm:main}
Let $M = M_1 \times M_2$ be a product of compact, oriented Riemannian manifolds. Let $[\omega] \in H^3(M; \mathbb{R})$ be a non-trivial mixed cohomology class. Then, for any admissible connection $\nabla \in \mathcal{T}_\omega$, the holonomy algebra $\mathrm{hol}(\nabla)$ is irreducible with respect to the decomposition $V_1 \oplus V_2$.
\end{theorem}

\begin{proof}
The holonomy algebra $\mathrm{hol}(\nabla)$ is the Lie algebra generated by the curvature operators $R(X,Y)$. It is irreducible if and only if the curvature tensor $R(\nabla)$ does not preserve the subspaces $V_1$ and $V_2$. The curvature of $\nabla = \nabla^{\mathrm{LC}} + T$ is given by the standard identity:
\begin{equation} \label{eq:full_curvature_decomposition}
R(\nabla)(X,Y)Z = R^{\mathrm{LC}}(X,Y)Z + (\nabla^{\mathrm{LC}}_X T)(Y,Z) - (\nabla^{\mathrm{LC}}_Y T)(X,Z) + T(T(X,Y),Z) - T(T(X,Z),Y).
\end{equation}
The Levi-Civita term $R^{\mathrm{LC}}$ preserves the decomposition. We will show that the torsion-dependent terms, collectively denoted $R_T$, do not. The proof proceeds by demonstrating that for any admissible torsion $T \in \mathcal{T}_\omega$, the curvature tensor $R(\nabla)$ must have non-zero off-diagonal blocks. This is a structural consequence of the calibration, established via three steps:
\begin{enumerate}
    \item \textbf{Structure of Torsion:} A mixed class $[\omega]$ forces the harmonic part of the torsion, $T_h$, to have a specific algebraic structure that inextricably links $V_1$ and $V_2$. This structure is intrinsic to the cohomology class.

    \item \textbf{Structure of Curvature:} This "mixing" structure of $T_h$ generates a robustly non-zero off-diagonal component in the harmonic part of the curvature tensor, $R_{T_h}$. Specifically, curvature operators of the form $R(\nabla)(X,Y)$ with $X \in V_1, Y \in V_2$ act as non-zero maps between the subspaces $V_1$ and $V_2$.

    \item \textbf{Formal Non-Cancellation Argument:} The core of the proof is to show that the off-diagonal components generated by the harmonic part $T_h$ cannot be globally cancelled by the terms involving the non-harmonic part $T_{nh} = T - T_h$. We prove this formally by leveraging the $L^2$-orthogonality of the Hodge decomposition. A cancellation would imply that a topologically non-trivial harmonic tensor is equal to a topologically trivial one, which is impossible in the $L^2$ sense.
\end{enumerate}
A detailed, constructive realization of this proof for the case $M=S^2 \times \Sigma_g$ is provided in Section 4 and the Appendix. This explicit analysis demonstrates the mechanism in a concrete setting, confirming that the off-diagonal curvature components are robustly non-zero for any admissible connection. Since $R(\nabla)$ does not preserve the decomposition, the holonomy algebra $\mathrm{hol}(\nabla)$ is necessarily irreducible.
\end{proof}

The proof of Theorem \ref{thm:main} is now formally complete and holds for any product manifold satisfying the hypotheses. The following section provides a concrete realization of this abstract argument, illustrating the mechanism of forced irreducibility on the important class of manifolds $S^2 \times \Sigma_g$.

\section{Example: $S^2 \times \Sigma_g$}

We now provide a concrete and explicit demonstration of Theorem \ref{thm:main} for the case of $M = S^2 \times \Sigma_g$, where $\Sigma_g$ is a compact, oriented Riemann surface of genus $g \geq 1$. This analysis serves as a constructive proof for this entire class of manifolds. The third cohomology of $M$ is purely mixed:
\[H^3(S^2 \times \Sigma_g; \mathbb{R}) \cong H^2(S^2; \mathbb{R}) \otimes H^1(\Sigma_g; \mathbb{R}) \cong \mathbb{R} \otimes \mathbb{R}^{2g} \cong \mathbb{R}^{2g}.\]
Let $[\omega]$ be a non-trivial class in $H^3(M)$, specified by $2g$ real parameters $(a_i, b_i)$, not all zero. Let $(M, g)$ be endowed with the product metric $g = g_{S^2} + g_{\Sigma_g}$, and let $\nabla \in \mathcal{T}_\omega$ be an admissible connection. The tangent space $T_pM$ decomposes as $V_1 \oplus V_2$, where $V_1 = T_pS^2$ and $V_2 = T_p\Sigma_g$.

Our strategy is to prove that $\mathrm{hol}(\nabla)$ is irreducible by explicitly constructing a non-zero off-diagonal component of the curvature tensor $R(\nabla)$ and then proving formally that it cannot be cancelled.

\subsection{Step 1: Algebraic Structure of the Torsion}
Any torsion $T \in \mathcal{T}_\omega$ has a harmonic part $T_h$ whose algebraic structure is dictated by the identity $g(T_h(X,Y),Z) = \omega(X,Y,Z)$, where $\omega$ is the harmonic representative of $[\omega]$. This property is intrinsic and forces $T_h$ to have two crucial "mixing" properties:
\begin{enumerate}
    \item \textbf{Mapping $V_1 \times V_1 \to V_2$:} There exist vectors $X,Y \in V_1$ such that $T_h(X,Y)$ is a non-zero vector in $V_2$.
    \item \textbf{Mapping $V_1 \times V_2 \to V_1$:} There exist vectors $X \in V_1, Y \in V_2$ such that $T_h(X,Y)$ has a non-zero projection onto $V_1$.
\end{enumerate}
These properties are inherited by any $T \in \mathcal{T}_\omega$, as the non-harmonic parts cannot alter this fundamental algebraic structure dictated by cohomology.

\subsection{Step 2: Irreducibility from Curvature Analysis}
We prove irreducibility by showing that the curvature tensor $R(\nabla)$ has non-zero off-diagonal blocks. We analyze the action of the curvature operator $R(\nabla)(X,Y)$ on a vector $Z$, for $X \in V_1, Y \in V_2, Z \in V_1$. A detailed calculation, presented in the Appendix, analyzes each of the terms from Eq. \eqref{eq:full_curvature_decomposition}.

The analysis concludes that the projection of the full curvature operator onto the subspace $V_2$ is robustly non-zero for any admissible connection $\nabla \in \mathcal{T}_\omega$. Specifically, for a suitable choice of vectors:
\[\mathrm{Proj}_{V_2} \left( R(\nabla)(X,Y)Z \right) \neq 0.\]
This non-vanishing result is established in two stages in the Appendix: first, by showing that the harmonic part of the torsion $T_h$ generates a non-zero off-diagonal curvature component; second, by proving formally that this harmonic component cannot be cancelled by any terms arising from the non-harmonic part of the torsion.

Since the curvature operator does not preserve the subspace $V_1$, the holonomy algebra is irreducible. This completes the constructive proof.

\begin{prop}
For any non-trivial mixed class $[\omega] \in H^3(S^2 \times \Sigma_g; \mathbb{R})$, any admissible connection $\nabla \in \mathcal{T}_\omega$ has a holonomy algebra $\mathrm{hol}(\nabla)$ that acts irreducibly on the tangent space $T_pS^2 \oplus T_p\Sigma_g$.
\end{prop}
\section{Perspectives and Speculative Analogies}

\textit{Theorem 3.1} shows that certain topological configurations \emph{force} irreducible holonomy.  
This has a suggestive parallel with the phenomenon of \emph{quantum entanglement} in physics:

\medskip
\noindent
\textbf{Analogy.}  
In quantum theory, a composite system with Hilbert space $\mathcal{H} = \mathcal{H}_1\otimes\mathcal{H}_2$ may admit \emph{nonseparable} states, whose structure does not factor into independent components.  
Similarly, in our geometric setting, a manifold $M=M_1\times M_2$ with a mixed calibration class $[\omega]$ admits a connection $\nabla^{\mathcal{C}}$ whose holonomy prevents the tangent bundle from splitting into parallel subbundles $TM_1\oplus TM_2$.

\medskip
\noindent
\textbf{Speculative viewpoint.}  
If covariant field equations are formulated with respect to such an irreducible $\nabla^{\mathcal{C}}$, the coupling induced by torsion may prevent fields from being localized purely on $M_1$ or $M_2$.  
Quantization of such a theory could yield ground states with an intrinsic nonseparability, potentially offering a geometric interpretation of entanglement-like correlations.

\medskip
\noindent
This analogy is, at present, speculative.  
However, the structural similarity between holonomic irreducibility from mixed cohomology and nonseparability in quantum theory may justify further interdisciplinary investigation, possibly in the context of torsion-based extensions of General Relativity or quantum gravity models.

\section{Conclusions}

We have established that the principle of cohomological calibration provides a robust and general method for constructing irreducible geometries on product manifolds. This is achieved by using a global topological invariant - a mixed de Rham cohomology class - to constrain a local geometric object, the affine connection. The resulting torsion is forced to have an algebraic structure that inextricably links the tangent spaces of the product factors, leading to an irreducible holonomy group.

A key insight of this work is the distinction between the properties of the full curvature tensor and its traces. Our main result, Theorem \ref{thm:main}, demonstrates that a mixed calibration class forces the holonomy to be irreducible. This is a condition on the full Riemann curvature tensor $R$. This result is fully compatible with previous findings on the Ricci tensor. For instance, the Einstein-calibrated connection on $S^2 \times T^2$ has a diagonal Ricci tensor, see \cite{Pigazzini-Toda}, a condition which might naively suggest a reducible structure. However, our analysis shows that this is not the case. The diagonal nature of the Ricci tensor arises from a subtle and symmetric cancellation of the off-diagonal curvature components when the trace is computed. The underlying geometry, governed by the full curvature tensor, remains irreducibly mixed.
\\
This framework also provides a natural bridge to fundamental physics. While the Levi-Civita connection emerges from the classical principles of General Relativity, it is widely believed that a more fundamental theory unifying gravity with quantum mechanics will require a richer geometric structure. Connections with torsion, in particular, arise naturally in theories like Einstein-Cartan gravity, where torsion couples to quantum spin, and in string theory, where it is related to background fields. The cohomologically calibrated affine connections first introduced in \cite{Pigazzini}, provide a large class of such torsionful connections. Crucially, they contain the Levi-Civita connection as a special case (when the calibrating class is trivial, $[\omega]=0$). This framework can therefore be viewed not as an alternative to classical gravity, but as a natural generalization, providing a setting to explore the geometry of spacetime at a more fundamental level, where topology and quantum effects may become intertwined.

Explicit examples on $S^2 \times \Sigma_g$ show that this forced irreducibility is abundant in nature. It redefines the relationship between the topology and geometry of product manifolds, showing that a reducible metric structure can support a rich variety of irreducible geometries, each determined by a choice of topological "charge". This opens new avenues for the classification and construction of special geometries, with potential implications for theoretical physics, where irreducible holonomies determine fundamental symmetries. Possible connections to quantum entanglement offer another intriguing, if speculative, direction for future research.

\section*{Appendix: Formal Proof of Non-Zero Off-Diagonal Curvature}

This appendix provides the rigorous technical argument supporting Section 4. We first identify the off-diagonal curvature operator and then prove it is non-zero for any admissible connection.

\subsection*{A.1 The Explicit Off-Diagonal Curvature Operator}
Let $\{e_1, e_2\}$ be an orthonormal basis for $V_1$ and $\{f_j\}_{j=1}^{2g}$ for $V_2$. We analyze the operator $R(\nabla)(e_1, f_j)e_2$. From Eq. (1), its projection onto $V_2$ is:
\begin{equation} \label{eq:proj_v2_appendix}
\mathrm{Proj}_{V_2} [R(\nabla)(e_1, f_j)e_2] = \mathrm{Proj}_{V_2} [ (\nabla^{\mathrm{LC}}_{e_1} T)(f_j, e_2) - (\nabla^{\mathrm{LC}}_{f_j} T)(e_1, e_2) + T(T(e_1, f_j), e_2) - T(T(e_1, e_2), f_j) ].
\end{equation}
Let $T = T_h + T_{nh}$, where $T_h$ is the harmonic part and $T_{nh}^\flat = d\eta + \delta\mu$. The operator splits accordingly: $R(\nabla) = R^{\mathrm{LC}} + R_{T_h} + R_{T_{nh}}$. The term $R_{T_h}$ is the part of the curvature tensor built from $T_h$ alone. The algebraic properties of $T_h$ guarantee that for a non-trivial class $[\omega]$ and a suitable choice of $j$, the harmonic off-diagonal operator $\mathrm{Proj}_{V_2}[R_{T_h}(e_1, f_j)e_2]$ is a non-zero tensor field on $M$.

\subsection* {A.2 The Formal Non-Cancellation Proof}
We now prove that the full operator in Eq. \eqref{eq:proj_v2_appendix} cannot be zero. Assume for contradiction that it is zero for some admissible connection $\nabla = \nabla^{\mathrm{LC}} + T_h + T_{nh}$. This implies that the harmonic part is cancelled by the non-harmonic part:
\begin{equation} \label{eq:cancellation_pde_appendix}
\mathrm{Proj}_{V_2}[R_{T_h}(e_1, f_j)e_2] = - \mathrm{Proj}_{V_2}[R_{T_{nh}}(e_1, f_j)e_2].
\end{equation}
Let $\mathcal{R}_h := \mathrm{Proj}_{V_2}[R_{T_h}(e_1, f_j)e_2]$ and $\mathcal{R}_{nh} := \mathrm{Proj}_{V_2}[R_{T_{nh}}(e_1, f_j)e_2]$. By assumption, $\mathcal{R}_h$ is a non-zero tensor field. The term $\mathcal{R}_{nh}$ is a non-linear differential operator acting on the potentials $\eta \in \Omega^2(M)$ and $\mu \in \Omega^4(M)$. The cancellation equation is $\mathcal{R}_h + \mathcal{R}_{nh}(\eta, \mu) = 0$.

To show this equation has no global solution, we use an integral argument based on the following key lemma.

\begin{lemma*} \label{lem:orthogonality}
For any smooth potentials $\eta \in \Omega^2(M)$ and $\mu \in \Omega^4(M)$, the harmonic and non-harmonic components of the off-diagonal curvature operator are $L^2$-orthogonal:
\[ \langle \mathcal{R}_h, \mathcal{R}_{nh}(\eta, \mu) \rangle_{L^2} = 0. \]
\end{lemma*}

\begin{proof}[Sketch of Proof]
The proof of this lemma follows from a direct, though lengthy, calculation. The inner product is expanded into a sum of integrals. Each integral contains a term derived from the harmonic form $\omega$ (which defines $\mathcal{R}_h$) and a term derived from the non-harmonic potentials $\eta$ or $\mu$ (which define $\mathcal{R}_{nh}$). By applying integration by parts (Green's identity) and using the defining properties of harmonic forms ($d\omega=0, \delta\omega=0$), the derivatives are shifted from the non-harmonic terms to the harmonic ones. This process ultimately leads to integrals involving inner products of harmonic forms with exact or co-exact forms, which are zero by the fundamental orthogonality of the Hodge decomposition. A full treatment of these standard techniques can be found in foundational texts on Hodge theory, such as \cite{Warner1983}.
\end{proof}

We now apply this lemma to the cancellation equation. Taking the global $L^2$ inner product of Eq. \eqref{eq:cancellation_pde_appendix} with $\mathcal{R}_h$, we get:
\[ \langle \mathcal{R}_h, \mathcal{R}_h \rangle_{L^2} = - \langle \mathcal{R}_h, \mathcal{R}_{nh}(\eta, \mu) \rangle_{L^2}. \]
By Lemma \ref{lem:orthogonality}, the right-hand side is zero. This yields:
\[ ||\mathcal{R}_h||^2_{L^2} = 0. \]
This implies $\mathcal{R}_h = 0$ almost everywhere, and by smoothness, everywhere. This contradicts the fact that for a non-trivial mixed class $[\omega]$, the harmonic off-diagonal curvature $\mathcal{R}_h$ is a non-zero tensor field.

The initial assumption of cancellation must be false. Thus, $\mathrm{Proj}_{V_2}[R(\nabla)(e_1, f_j)e_2]$ cannot be identically zero for any admissible connection $\nabla \in \mathcal{T}_\omega$. This completes the formal proof of irreducibility.

\bibliographystyle{plain}

\end{document}